\documentclass[a4paper,leqno,11pt]{article}

\usepackage{amsmath,amssymb,amsthm,%showkeys,
verbatim%,refcheck
}

\usepackage{tikz}
\tikzstyle{nodo}=[circle,draw,fill,inner sep=0pt,minimum size=%
\widthof{r}]
\tikzstyle{infinito}=[circle,inner sep=0pt,minimum size=0mm]

\newcommand\suchthat{:\,}
\newcommand\R{{\mathbb R}}

\newcommand\elevel{{\mathcal E}}

\newcommand\G{\mathcal G}
\newcommand\K{\mathcal K}
\newcommand\HH{\mathcal H}

\newcommand\vv{\textsc{v}}

\newcommand\diam{\mathop{\rm diam}}

\newcommand\eps{\varepsilon}

\newtheorem{theorem}{Theorem}[section]
\newtheorem{proposition}[theorem]{Proposition}
\newtheorem{lemma}[theorem]{Lemma}
\newtheorem{corollary}[theorem]{Corollary}

\theoremstyle{remark}
\newtheorem{remark}[theorem]{Remark}
\newtheorem*{remark*}{Remark}

\theoremstyle{definition}
\newtheorem{definition}[theorem]{Definition}

\date{}

\title{Threshold phenomena and existence results\\
for NLS ground states on graphs}

\author{Riccardo Adami\thanks{Author partially supported by the FIRB 2012 project ``Dispersive dynamics: Fourier
Analysis and Variational Methods".}, Enrico Serra\thanks{Author
partially supported by the PRIN 2012 project ``Aspetti
variazionali e perturbativi nei problemi differenziali
nonlineari".}, Paolo Tilli \\ \ \\{\small  Dipartimento di Scienze
Matematiche ``G.L. Lagrange'', Politecnico di Torino } \\ {\small
Corso Duca degli Abruzzi, 24, 10129 Torino, Italy}}

\begin{document}

\maketitle

\begin{abstract}
We investigate the existence of ground states of prescribed mass,
for the nonlinear Schr\"odinger energy on a noncompact metric graph $\G$.
While in some cases the topology of $\G$ may rule out or, on the contrary,
guarantee the existence of ground states of any given mass, in general
also some \emph{metric} properties of $\G$, and their quantitative relation
with the actual value of the prescribed mass, are relevant as concerns the
existence of ground states.
This may give rise to interesting phase transitions from nonexistence to existence
of ground states, when a certain quantity reaches a critical threshold.
\end{abstract}

\smallskip

\noindent{\small Keywords: Minimization, metric graphs, rearrangement,
  nonlinear Schr\"odinger \\ \hbox{} \hskip 1.65cm Equation.}

\section{Introduction}
In this paper we carry on our investigation, initiated in \cite{ast}, concerning
the existence of ground states for the NLS energy functional
\begin{equation}
\label{NLSe}
E (u,\G)
  =  \frac 1 2 \| u' \|^2_{L^2 (\G)}
- \frac 1 p  \| u \|^p_{L^p (\G)}
=\frac 1 2 \int_\G |u'|^2dx
-\frac 1 p \int_\G |u|^pdx
\end{equation}
on a noncompact
metric graph $\G$, under the
{\em mass constraint}
\begin{equation}
\label{mass}
\| u \|^2_{L^2 (\G)} \ = \ \mu.
\end{equation}
Throughout the paper, the exponent $p\in (2,6)$ is \emph{fixed}, %within the bounds
while the mass $\mu$ is a parameter of the problem.
The domain $\G$ is a \emph{connected metric graph}, that is, a connected graph whose
edges are (possibly half-infinite) segments of line, joined at their endpoints (the vertices
of $\G$) according to the topology of the graph. Each edge $e$, after choosing a coordinate $x_e$ on it,
can be regarded either as an interval $[0,\ell_e]$, or as a positive half-line $[0,+\infty)$ (in this case
the edge is attached to $\G$ at $x_e=0$), and the spaces $L^r(\G)$, $H^1(\G)$ etc. can be
defined in a natural way (we refer to \cite{ast} for more details). Endowing $\G$ with the shortest path distance,
one obtains a locally compact metric space: when $\G$ consists of just one unbounded edge, for instance,
one obtains $\R^+$,
while $\R$ corresponds to two unbounded edges (for other examples, see Figures~\ref{bubbles}--\ref{3-fork}).

In this framework, by a ``ground state of mass $\mu$'' we mean a solution to
the minimization problem
\begin{equation}
\label{minprob}
\min_{u\in H^1_\mu(\G)} E(u,\G),\qquad
H^1_\mu(\G):=\left\{ u\in H^1(\G)\suchthat\Vert u\Vert_{L^2(\G)}^2=\mu\right\}.
\end{equation}
Since existence of solutions is trivial when $\G$ is compact, we
will always assume that $\G$ is \emph{noncompact} or,
equivalently, that at least one edge of $\G$ is unbounded.
Moreover, when dealing with a ground state $u$, we will always
assume that $u>0$ (up to a constant phase, as shown in \cite{ast},
this is not restrictive: for this reason, we only consider real
valued functions). Finally, we mention that any ground  state $u\in
H^1_\mu(\G)$ satisfies, for a suitable Lagrange multiplier
$\lambda$, the nonlinear equation
\begin{equation}
\label{euleroforte} u''+u |u|^{p-2}=\lambda u
\end{equation}
on every edge  of $\G$, coupled with a homogeneous \emph{Kirchhoff
condition} at every vertex of $\G$ (see \cite{ast} for more details).

The main results of \cite{ast} can be summarized as follows. On the one hand
it was proved that, apart from certain particular cases,
a topological condition on $\G$ called ``assumption (H)''
prevents
the existence of ground states for every value of  $\mu$.
If all the $\infty$-points of the half-lines of $\G$
are regarded
as a single vertex, %as in a one-point compactification
this assumption  takes the form
\begin{equation*}
\tag{H}
\text{%assumption (H):
$\quad\G$, as a graph, can be covered by cycles}
\end{equation*}
(observe that a noncompact $\G$
satisfying (H) must have \emph{at least two half-lines}).
On the other hand,
the case where $\G$ consists of two half-lines and
a finite interval, all emanating from the same vertex,
was studied in detail, and it was proved that a ground state
does exist for every $\mu$. Indeed this topology, namely a real line with an interval attached
at one endpoint, is the simplest one that violates
assumption (H), among graphs with at least two half-lines.

Thus, regardless of the prescribed mass $\mu$,
certain topologies
of $\G$ rule out the existence of ground states, while, in view of the mentioned example,
other topologies may guarantee their existence.

\smallskip

The main purpose of this paper is to provide general sufficient conditions on $\G$, not only topological but
also of a \emph{metric} nature, to guarantee the existence of a ground state of prescribed mass $\mu$.
%,  is to investigate how a quantitative interplay
%between the given mass $\mu$ and the
%\emph{metric properties}
%of $\G$, besides its topology, may indeed be relevant (in one direction or another)
%as regards the existence of ground states with prescribed mass $\mu$.
On a general ground
the outcome is that, given the topology of $\G$,  the existence
of ground states will strongly depend on the interplay between
the metric properties of $\G$ (i.e. the actual
lengths of its bounded edges) and the value of $\mu$. A typical example
%of this phenomenon
is Proposition~\ref{corbaffolungo}: if $\G$ has a terminal edge of length $\ell$,
and the product $\ell\mu^\beta$  is large enough
($\beta=\frac {p-2}{6-p}$), then $\G$ admits a ground state.
Observe that $\mu$ and the metric properties
of $\G$ are related in a natural way, due to how the minimization problem \eqref{minprob}
scales, under homotheties of $\G$ (see Remark~\ref{remscal}): in particular, $\mu^{-\beta}$ scales
as a \emph{length}.

Central to our investigation is the ground-state energy level
\begin{equation}
\label{defiG}
\elevel_\G(\mu):=\inf\left\{E(u,\G)\suchthat u\in H^1_\mu(\G)\right\},\quad \mu\geq 0,
\end{equation}
regarded as a function of $\mu$, for fixed $\G$. In \cite{ast}  it was proved that
\begin{equation}
\label{trapped}
%-\,\frac 1 2 \theta_p (2\mu)^{\frac {p+2}{6-p}}=
\elevel_{\R^+}(\mu)\,
\leq\,\elevel_\G(\mu)\,\leq\, \elevel_{\R}(\mu)
%=-\theta_p \mu^{\frac {p+2}{6-p}}
\end{equation}
for every noncompact $\G$. The quantities on the right and on the left are, respectively,
the energy level of a \emph{soliton} of mass $\mu$ on the real line, and the energy level
of a \emph{half-soliton} of mass $\mu$ on the positive half-line. Thus, in a sense, $\R^+$ and $\R$
are extremal, among noncompact graphs, as regards the ground-state energy level.

As we shall prove in Theorem~\ref{main1}, if the second inequality in \eqref{trapped}
is \emph{strict}, then $\G$ admits a ground state of mass $\mu$.
In other words, the existence of a minimizer for \eqref{minprob} is guaranteed,
as soon as one constructs a function $u\in H^1_\mu(\G)$ with an energy level
\emph{not higher}
than the energy level of a soliton on the real line (Corollary~\ref{oper}):
this %sufficient condition
is quite effective in the applications (see the examples at the end of Section \ref{exi})
since, starting from a soliton $\phi_\mu$ on $\R$, one may try to ``cut pieces''
of $\phi_\mu$ and, possibly after monotone rearrangements, paste them on the graph $\G$,
to obtain a competitor $u\in H^1_\mu(\G)$ with a lower energy.

This result also entails that, whenever $\G$ admits no ground state of mass $\mu$, its
energy level $\elevel_\G(\mu)$, though not achieved by
any function $u\in H^1_\mu(\G)$, is necessarily equal to that  of the soliton.

We also prove (Theorem~\ref{subadd}) that $\elevel_\G(\mu)$ is \emph{strictly} subadditive and concave.
This information allows us to completely characterize
the behavior of any minimizing sequence $\{u_n\}$, relative to \eqref{minprob}, by a dichotomy principle
(Theorem~\ref{ps}): if $u_n\rightharpoonup u$ in $H^1(\G)$, then
\emph{either} $u_n\to u$  \emph{strongly} (and $u\in H^1_\mu(\G)$ is a minimizer), \emph{or}
$u\equiv 0$ (and $u_n$, in this case, loses \emph{all} of its mass at infinity).

In general, the weak compactness of minimizing sequences is guaranteed by a set
of new a priori estimates (Lemma~\ref{propgen}), which are \emph{universal}
in the following sense:
they are given in terms of powers of $\mu$ and numerical constants $C$, that
depend only on the exponent $p$ and \emph{not} on the particular graph $\G$.
%(of course, a ground state -if any exists- must obey these estimates).
This, in turn, is a consequence
of the universal Gagliardo-Nirenberg inequality \eqref{GN-universal}.

As regards the existence of ground states, interesting \emph{phase transitions} may occur
if, for fixed $\mu$, one alters the metric properties of $\G$ without changing its topology
(and, dually, the same may occur if $\mu$ is altered while $\G$ is unchanged). Here
we start the investigation of this phenomenon by a case study, namely when $\G$ is made up
of $N$ half-lines ($N\!\geq\! 2$) and a pendant edge of length $\ell$, all emanating from the same vertex.
When $N=2$ it is known from \cite{ast} that a ground state exists for every
$\ell,\mu>0$. When $N>2$, we prove (Theorem~\ref{teoksemirette}) that there exists a number $C^*>0$,
depending only on $p$ and $N$, such that a ground state exists if and only if $\ell \mu^{\beta}\geq C^*$ ($\beta =\frac{p-2}{6-p}$).
Of course, when $\ell \mu^{\beta}$ is large enough, the existence of a ground state follows
from Proposition~\ref{corbaffolungo}: the real point is that %, when $\G$ has this particular topology,
a \emph{sharp phase transition}, from nonexistence to existence of a ground state, occurs at $C^*$
(if $\mu$ is fixed, the phase transition occurs at the critical length $\ell^*=C^*\mu^{-\beta}$ while,
if $\ell$ is fixed, it occurs at the critical mass $(\ell/C^*)^\beta$). It is an open problem to
establish if such a sharp transition is a general fact (a quite strong reinforcement of Proposition~\ref{corbaffolungo}),
or if (and to what extent) it is peculiar to this example
(our proof builds on Theorem~\ref{stab}, a stability result of general validity, but the
particular structure of $\G$ is somehow exploited, to prove that ground states persist when $\ell$ is enlarged).
We point out that this
topology is the simplest one that violates assumption (H), among all metric graphs  with $N$ half-lines.

Finally, an issue that remained open from \cite{ast} was to establish whether a metric graph
$\G$, having just \emph{one half-line}, always admits a ground state (in this case
assumption (H) is automatically violated, and
ground states may indeed exist: see Figures~\ref{tadpole}, \ref{3-fork}).
As a metric graph, any such $\G$ is just a compact perturbation of $\R^+$, hence
a competitor $u\in H^1_\mu(\G)$ might well exist, quite similar in shape to a half-soliton on $\R^+$ and
with a comparable energy: the energy level $\elevel_\G(\mu)$ would then
be closer to its lower bound than to its upper bound in \eqref{trapped}, and a ground
state would then exist by virtue of Theorem~\ref{main1}.

In Section~\ref{sectionscopino}, however, we show how counterexamples can be constructed,
and the idea underlying the proof highlights a new interesting phenomenon:
%it turns out that,
if $\G$ consists of a compact core $\K$ attached to one \emph{single} half-line,
ground states with a fixed mass cannot exist, if $\K$ has a small diameter and a large total length
(see Theorem~\ref{teoscopino} and Corollary~\ref{corscopino}).
%a crucial role
%is played by the embedding $H^1(\K)\hookrightarrow L^\infty(\K)$: if
%the embedding constant $C_\K$ is small enough (depending on  $\mu$), then $\G$ admits no ground state.
%A sufficient condition for $C_\K$ to be small is that $\K$ has a small diameter but a large total length,
%thus a concrete example is a $\G$ made up of one half-line and several short segments attached to it, at one endpoint
%(see Corollary~\ref{corscopino}).

\medskip

The subject of dynamics on quantum graphs is now recognized as a relevant issue.
%in nonlinear analysis.
Starting from seminal works \cite{ali,vonbelow},
nonlinear propagation on networks has been proposed in different
contexts (\cite{bona,mugnolo,matrasulov}). The rigorous study of the NLS
equation
on graphs started with \cite{acfn1}, while the problem of minimizing
the NLS energy %associated to the NLS
was first faced in
\cite{acfn2}, but limited to the case of star graphs.
Bridge-type graphs
were treated in
\cite{ast2},
while the first general results on noncompact graphs are contained in  \cite{ast}.
The problem of energy minimization on graphs
with a mixed propagation (i.e. linear on unbounded edges, while nonlinear
in the compact core), well-known in the physical literature (see
e.g. \cite{smi}), is considered in \cite{tentarelli}.
For general definitions and results on metric graphs,
we refer to \cite{berkolaiko,exner,post}.

\section{General properties of minimizers}

In this section we establish several new a priori estimates for ground states
of prescribed mass. An interesting feature is that these estimates
do \emph{not} depend on the particular structure of the
graph $\G$ (provided it is not compact). Throughout, we denote by $C_p$
(or simply by $C$) a generic positive constant that depends only on the exponent
$p$.

The starting point is the following Gagliardo-Nirenberg inequality.

\begin{proposition}[Universal G-N inequality] There exists %a universal constant
$C_p>0$ such that
\begin{equation}
\label{GN-universal}
\Vert u\Vert_{L^p(\G)}^p\leq
C_p \Vert u\Vert_{L^2(\G)}^{\frac p 2 + 1}
\Vert u'\Vert_{L^2(\G)}^{\frac p 2 - 1},
\end{equation}
for every $u\in H^1(\G)$ and every noncompact metric graph $\G$.
\end{proposition}
\begin{proof}
When $\G=\R^+$ is a half-line, \eqref{GN-universal} is well known (see \cite{dell}).
In general, if $\G$ is noncompact and $u\in H^1(\G)$, one may pass from $u$
to its decreasing rearrangement $u^*\in H^1(\R^+)$
(see \cite{ast,friedlander}). This passage preserves all the
$L^r$ norms, and does not increase the $L^2$ norm of $u'$ (see \cite{ast}). As a consequence
\eqref{GN-universal}, which is true for $u^*$ when $\G=\R^+$, is true for $u$ as well,
with the same constant.
\end{proof}
If $\G$ is compact then \eqref{GN-universal}
does not hold (in this form), as one can see letting $u\equiv 1$.
For noncompact $\G$, however, in
exactly the same way one can prove the validity of
\begin{equation}
\label{interpol}
\Vert u\Vert_{L^\infty(\G)}^2 \leq
C \Vert u\Vert_{L^2(\G)}
\Vert u'\Vert_{L^2(\G)},
\end{equation}
with $C>0$ independent of $\G$ ($C=2$ will do).

An important role in the sequel
is played by the ground states on the real line, known as
 \emph{solitons} (see \cite{cazenave,zakharov}).
\begin{remark}[Solitons]\label{remsol} When $\G=\R$ the solutions to \eqref{minprob},
called solitons,
are unique up to translations
and a change of sign. We denote by  $\phi_\mu$ (the dependence on $p$ being understood)
the positive soliton of mass $\mu$ centered at the origin, whose dependence on $\mu$
is given by the scaling rule
\begin{equation}
\label{scalingrule}
\phi_\mu(x)=
\mu^\alpha \phi_1\bigl ( \mu^\beta x\bigr ),\quad
\alpha=\frac 2{6-p},\quad
\beta=\frac {p-2} {6-p}
\end{equation}
where $\phi_1(x)=C_p \mathop{\rm sech}(c_p x)^{\alpha/\beta}$ with $C_p,c_p>0$
(also note that $\alpha,\beta>0$ since $p\in (2,6)$). Then a direct computation shows that
\begin{equation}
\label{energiasol}
\elevel_\R(\mu)=
E(\phi_\mu,\R)=-\theta_p \mu^{2\beta+1},\quad
\theta_p:=-E(\phi_1,\R)>0.
\end{equation}
When $\G=\R^+$, the unique positive ground state is the ``half soliton'', i.e. $\phi_{2\mu}$
restricted to $\R^+$, so that now
\begin{equation}
\label{energiamezzosol}
\elevel_{\R^+}(\mu)=E(\phi_{2\mu},\R^+)=\frac 1 2E(\phi_{2\mu},\R) =-2^{2\beta}\theta_p \mu^{2\beta+1}
\end{equation}
with $\theta_p$ as above. Then, we see that \eqref{trapped} takes the concrete form
\begin{equation}
\label{trapped2}
-2^{2\beta}\theta_p \mu^{2\beta+1}\,\leq
\inf_{u\in H^1_\mu(\G)} E(u,\G)\,\leq\, %\inf_{u\in H^1_\mu(\R)} E(u,\R)=
-\theta_p \mu^{2\beta+1}.
\end{equation}
This notation concerning solitons, and in particular the exponents $\alpha,\beta$ defined in
\eqref{scalingrule}, will be used systematically throughout the paper, without further reference.
\end{remark}

\begin{remark}[Scaling]\label{remscal} If $u\in H^1(\G)$, the quantities
\[
\mu^{-2\beta-1}\Vert u'\Vert_{L^2(\G)}^2,\quad
\mu^{-2\beta-1}\Vert u\Vert_{L^p(\G)}^p,
\quad
\mu^{-\beta-1}\Vert u\Vert_{L^\infty(\G)}^2,
\]
where $\mu:=\Vert u\Vert_{L^2(\G)}^2$,
are \emph{invariant}, if one dilates $\G$ and and rescales $u$ according to
\[
\G\mapsto t^{-\beta}\G,\quad
u(\cdot)\mapsto t^{\alpha} u(t^\beta\cdot),\qquad\text{($t>0$)}
\]
(the same scaling rule as in \eqref{scalingrule}, for solitons).
Clearly, also the normalized energy $\mu^{-2\beta-1}E(u,\G)$ is invariant, while
the mass $\Vert u\Vert_{L^2}^2$ passes from $\mu$ to $t\mu$. Therefore, the minimization problem
\eqref{minprob} (with mass constraint $\mu$) is equivalent to one with any desired mass constraint
(e.g. $\mu=1$)
on another graph, homothetic to $\G$. Similarly, any characteristic length $\ell$ on $\G$
(e.g. the diameter of its compact core, or the length of a given edge) transforms as
$\ell\mapsto t^{-\beta}\ell$, so that the quantity $\mu^\beta\ell$ is \emph{scale invariant}.
\end{remark}

\begin{definition}
If $\G$ is a metric graph, we define its \emph{compact core} $\K$ as
the metric graph obtained from $\G$ by removing every unbounded edge
(half-line).
\end{definition}
As a metric space, $\K$ is obtained from $\G$ by removing the \emph{interior}
of every half-line (so that the origin of the half-lines still belong to $\K$).
It is clear that the compact core $\K$ is compact: moreover, when
$\G$ is connected (as we always assume throughout) $\K$ is
also connected, since every half-line is a terminal edge for $\G$.
If $\G$ consists only of half-lines, all sharing the same vertex $\vv$ as their origin,
then $\K$ (as a graph) has no edge  and has $\vv$ as its only vertex: as a metric space,
in this case $\K$ consists of one isolated point.

\begin{proposition}\label{propmaxcc}
Assume $\G$ is a metric graph with at least two half-lines, and let $u\in H^1_\mu(\G)$
for some $\mu>0$. Then
\begin{itemize}
\item[(i)] If the $L^\infty$ norm of $u$ is attained on a half-line of $\G$, then
\begin{equation}
\label{peggiodisolitone}
E(u,\G)\geq E(\phi_\mu,\R),
\end{equation}
and the inequality is \emph{strict}, unless $\G$ is isometric to the real line
and $u$ is a translate of the soliton $\phi_\mu$.
\item[(ii)] If $u$ is a minimizer, then its $L^\infty$ norm is attained on the compact core of $\G$,
unless $\G$ is isometric to $\R$ and $u$ is a soliton.
\end{itemize}
\end{proposition}
\begin{proof}\emph{(i)} Replacing $u$ with $|u|$, we may assume that $u\geq 0$ on $\G$.
Let $y$ be a point on a half-line ${\mathcal H}$ where $u(y)=\Vert u\Vert_{L^\infty}$, and
observe that $u(y)>0$ since $u\in H^1_\mu(\G)$. Since $u(x)\to 0$ as $x\to\infty$ along every half-line,
if $0<t<u(y)$ then $u^{-1}(t)$ has \emph{at least two} preimages in $\G$: one on ${\mathcal H}$, between
$y$ and the $\infty$-point  of ${\mathcal H}$, and another one along any path $P$ that joins $y$ to
the $\infty$-point of a half-line other than ${\mathcal H}$. Then, if $\widehat u$ denotes the symmetric
rearrangement of $u$ on $\R$, from Prop.~3.1 of \cite{ast} we have $E(u,\G)\geq E(\widehat u,\R)$ and
\eqref{peggiodisolitone} follows, since $\phi_\mu$ is a minimizer in $H^1_\mu(\R)$
(unique, up to a translation and a sign change).  If equality occurs, then necessarily $\widehat u=\phi_\mu$ and,
since $u$ and $\widehat u$ are equimeasurable, we deduce that $u^{-1}(t)$ has \emph{exactly two} preimages
on $\G$, for every $t\in (0,u(y))$: then, if $P$ is a path of the kind described above, it follows that
${\mathcal H}\cup P$ covers $\G$, and the claim follows.

\noindent\emph{(ii)} This part follows immediately from (i), recalling \eqref{energiasol}.
\end{proof}

In the next result we show that all the relevant quantities are controlled in terms of $\mu$
(independently of $\G$) in suitable sublevel sets of the energy $E$.

\begin{lemma}\label{propgen}
Let $\G$ be a non-compact graph, and let $u\in H^1_\mu(\G)$ be
such that
\begin{equation}
\label{Eneg} E(u,\G)\leq \frac 1 2 \inf_{v\in H^1_\mu(\G)} E(v,\G).
\end{equation}
 Then
\begin{alignat}{2}
\label{st29} C_p^{-1}\mu^{2\beta+1}\, &\leq\, \int_\G |u'|^2\,
dx&&\leq C_p \mu^{2\beta+1}, %\qquad\text{(note: $2\beta+1=(p+2)/(6-p)$).}
\\
\label{stimaB} C_p^{-1}\mu^{2\beta+1}\, &\leq\, \int_\G |u|^p\,
dx&&\leq C_p \mu^{2\beta+1},
%\qquad\text{(note: $2\beta+1=(p+2)/(6-p)$).}
\\
\label{st30} C_p^{-1}\mu^{\beta+1}\, &\leq\, \,\Vert
u\Vert_{L^\infty(\G)}^2&\,&\leq\, C_p\mu^{\beta+1},
\end{alignat}
for some constant $C_p>0$ depending only on $p$.
\end{lemma}
\begin{remark}
\label{remgen}
By \eqref{trapped2} and \eqref{energiasol}, the infimum in \eqref{Eneg} is strictly negative,
hence \eqref{Eneg} is certainly satisfied when $u$ is a ground state or, more generally,
when $u$ belongs to a minimizing sequence $\{u_n\}$ and $n$ is large enough.
\end{remark}

\begin{proof}
By Remark~\ref{remscal}, we may assume that $\mu=1$. In this case, \eqref{GN-universal} takes the form
\begin{equation}
\label{GNnor}
 V\leq C_p T^{\frac {p-2}{4}}
\end{equation}
where $T,V$ denote, respectively, the two integrals in
\eqref{st29} and \eqref{stimaB}. On the other hand, \eqref{Eneg}
combined with \eqref{trapped2} gives
\begin{equation}
\label{TmenV}
\frac12 T-\frac1pV=E(u,\G)\leq - \,\,\frac{\,\theta_p}2<0.
\end{equation}
In particular $T<\frac2p V$ which, combined with \eqref{GNnor}, gives
$T\leq C$, i.e. the second estimate in \eqref{st29} when $\mu=1$. Then $V\leq C$
(the second part of \eqref{stimaB}) follows from \eqref{GNnor}. Finally, the
second inequality in \eqref{st30}
follows from \eqref{interpol}.
Now we prove the estimates from below. Since $T\geq 0$, \eqref{TmenV} gives $-V\leq -p\theta_p/2$,
i.e. $V\geq C^{-1}$. This, combined with \eqref{GNnor}, gives $T\geq C^{-1}$. Finally,
since $V\leq \mu \Vert u\Vert_{L^\infty(\G)}^{p-2}=\Vert u\Vert_{L^\infty(\G)}^{p-2}$,
the first part of \eqref{st30} follows from $V\geq C^{-1}$.
\end{proof}

\begin{corollary}\label{corlambda} Let $\G$ be a noncompact graph, and let $u\in H^1_\mu(\G)$ be
a ground state. Then the Lagrange multiplier $\lambda$ in \eqref{euleroforte} satisfies
\begin{equation}
\label{stimalambda}
C_p^{-1}\mu^{2\beta} \leq \lambda\leq C_p\mu^{2\beta}.
\end{equation}
Moreover, the restriction of $u$ to any half-line of $\G$ takes the form
\begin{equation}
\label{uhl}
u(x)= \phi_m(x+y),\quad x\geq 0,
\end{equation}
where $y\in \R$ depends on the half-line while $m$ is common to all the half-lines, and satisfies
\begin{equation}
\label{stimaMuniv}
C_p^{-1} \mu \leq m\leq C_p\mu.
\end{equation}
\end{corollary}
\begin{proof}
A ground state $u$ satisfies the Euler-Lagrange equation
\[
\int_\G \left(-u'\eta'+u |u|^{p-2}\eta\right)\,dx=\lambda
\int_\G u\eta\,dx\quad
\forall\eta\in H^1(\G),
\]
where $\lambda$ is the same as in \eqref{euleroforte} (see~\cite{ast}).
Testing  with $\eta=u$ yields
\[
-\int_{\G} |u'|^2\,dx+\int_{\G}|u|^p=\lambda \int_{\G} |u|^2\,dx=\lambda\mu,
\]
and \eqref{stimalambda} follows from \eqref{st29} and \eqref{stimaB}.
Moreover, once $\lambda$ is fixed, it is well known that any solution in $L^2(\R^+)$
of \eqref{euleroforte}
is necessarily a (portion of) soliton
as in \eqref{uhl}. Since, by \eqref{scalingrule}, $\lambda=C_p m^{2\beta}$, estimate \eqref{stimaMuniv} follows
from \eqref{stimalambda}.
\end{proof}
When several half-lines originate in the \emph{same vertex} of $\G$, the structure of
a ground state is even more rigid.
\begin{theorem}\label{teocode}
Assume that $\G$ is not homeomorphic to $\R$, and that $N$ half-lines ($N\geq 2$) emanate from the same vertex $\vv$.
Then,
along each of these $N$ half-lines,
any ground state $u$ takes the form \eqref{uhl} with the \emph{same, nonnegative value} of $y$.
\end{theorem}
\begin{proof}
Let us denote by $\HH_i$ ($1\!\leq\! i\!\leq\! N$) the $N$ half-lines originating in $\vv$.
By
Corollary~\ref{corlambda},
on each $\HH_i$ any ground state
$u(x)$
 takes the form
\eqref{uhl},
with the same $m$  and a shift  $y=y_i$
that may depend on $i$. In fact,
since $u$ is continuous at $\vv$
and the soliton $\phi_m(x)$ is even and radially decreasing, the absolute value $|y_i|$ is independent
of $i$, being
determined
by the condition $\phi_m(\pm y)=u(\vv)$. Therefore, it suffices to prove that $y_i\geq 0$
for every $i$.
Assuming, for instance, that $y_1<0$, we shall find a contradiction, by building
a family of competitors with a lower energy level than $u$.

Consider the two following
subgraphs of $\G$:
$\G_1:=\HH_1\cup \HH_2$, which is isometric to $\R$, and $\G_2$, obtained from
$\G$ by removing the edges $\HH_1$ and $\HH_2$
($\G_2$ has at least one edge, otherwise
$\G$ would be isometric to $\R$). We have $\G_1\cap \G_2=\{\vv\}$,
and we can split the ground state $u$ as $(u_1,u_2)$, with $u_i\in H^1(\G_i)$ satisfying the continuity
condition
\begin{equation}
\label{eqjoint}
u_1(\vv)=u_2(\vv).
\end{equation}
We also let $\mu_i=\Vert u_i\Vert^2_{L^2(\G_i)}$, so that $\mu_1+\mu_2=\mu$, the total
mass of $u$ (since $u>0$ on $\G$, we have $\mu_1,\mu_2>0$). Finally,
we observe that
\begin{equation}
\label{assurdo}
0=\inf_{\G_1} u_1< u_1(\vv) < \max_{\G_1} u_1,
\end{equation}
the second inequality being a direct consequence of the assumption $y_1<0$.
Now, consider the two functions $v_i^\eps\in H^1(\G_i)$ ($1\!\leq\! i\!\leq\! 2$) defined by
\begin{equation}
\label{defvv}
v_1^\eps(x):=\left(1+\eps\right)^{\frac 1 2} u_1(x),\qquad
v_2^\eps(x):=\left(1-\eps\mu_1/\mu_2\right)^{\frac 1 2} u_2(x),
\end{equation}
where $\eps$ is a small parameter. Clearly $\Vert v_1^\eps\Vert_{L^2(\G_1)}^2+\Vert v_2^\eps\Vert_{L^2(\G_2)}^2=\mu$,
but due to \eqref{eqjoint}, when $\eps\not=0$
we have $v_1^\eps(\vv)\not=v_2^\eps(\vv)$.
If $|\eps|$ is small enough, however, \eqref{assurdo}, \eqref{eqjoint} and \eqref{defvv} entail that
\[
0=\inf_{\G_1} v^\eps_1 <v^\eps_2(\vv)<\max_{\G_1} v^\eps_1.
\]
Therefore, since $\G_1$ is isometric to $\R$, we can
shift $v_1^\eps$ on $\G_1$ by a proper amount, in such a way
that the shifted function (still denoted by $v_1^\eps$ for simplicity)
satisfies $v_1^\eps(\vv)=v_2^\eps(\vv)$.
 Now
$v^\eps_1$ and $v^\eps_2$ can be glued together on $\G$, thus obtaining
a function in $H^1_\mu(\G)$ whose energy is given by
\[
\begin{split}
f(\eps):=
E(v_1^\eps,\G_1)+E(v_2^\eps,\G_2)=
\frac {1+\eps}{2}\int_{\G_1} |u_1'|^2\,dx
-
\frac {(1+\eps)^{\frac p 2}}{p}\int_{\G_1} |u_1|^p\,dx\\
+\frac {1-\eps\mu_1/\mu_2}{2}\int_{\G_2} |u_2'|^2\,dx
-
\frac {(1-\eps\mu_1/\mu_2)^{\frac p 2}}{p}\int_{\G_2} |u_2|^p\,dx
\end{split}
\]
(we can regard $v^\eps_1$ as originally defined in \eqref{defvv}, since energy
is shift invariant on $\G_1$).
Since clearly $f''(\eps)<0$ when $|\eps|$ is small,
$f$ cannot have a local minimum at $\eps=0$,
but this is a contradiction since $f(0)=E(u,\G)$ and $u$ is a ground state.
\end{proof}

It is an open problem to establish whether the previous result is still true when $N = 1$.

\section{Existence results}
\label{exi}

In this section we investigate the behavior of minimizing sequences on a generic non compact graph. We start with
a general concavity property for the energy level function in \eqref{defiG}, that holds on any noncompact graph.

\begin{theorem}[Concavity]
\label{subadd} The function $\elevel_\G:[0,\infty)\to[0,-\infty)$, as defined in \eqref{defiG}, is
\emph{strictly concave} and subadditive.
\end{theorem}

\begin{proof} For $u\in H^1(\G)$, we set $V_u=\int_\G |u|^p dx$ and define
\begin{equation}
\label{defU}
U:=\left\{ u\in H^1(\G)\suchthat\int_\G |u|^2dx=1,\quad \mu^{\frac p 2}V_u\geq C_p^{-1}\mu^{2\beta+1}\right\},
\end{equation}
where $C_p^{-1}$ is the same as in \eqref{stimaB}. Then, we consider the family of concave functions
\[
f_u(\mu):=E\left (\sqrt{\mu}\, u,\G\right)=\frac \mu 2 \int_\G |u'|^2\,dx-\frac {\mu^{\frac p 2}}{p} V_u,\quad \mu\geq 0,\quad
u\in U.
\]
By Remark~\ref{remgen}, the value of $\elevel_\G(\mu)$ is unaltered, if the infimum in \eqref{defiG} is further
restricted to functions
satisfying the lower bound in \eqref{stimaB} or, which is the same, to functions of the form
$\sqrt\mu\, u$ with $u\in U$. Therefore, we have
\[
\elevel_\G(\mu):=\inf_{u\in U} f_u(\mu),\quad
\mu\geq 0,
\]
so that $\elevel_\G$ inherits concavity from the $f_u$'s. Now, since
$f_u''(\mu)=-c_p \mu^{\frac p 2-2}V_u$ ($c_p\!>\!0$), we see
from \eqref{defU} that $f_u''(\mu)\leq -c_pC_p^{-1}\mu^{2\beta-1}$,
so that
the strict concavity of $f_u$ on every interval $[a,b]\subset (0,\infty)$ is \emph{uniform} in $u$.
Hence
$\elevel_\G$ is \emph{strictly} concave, and also strictly subadditive since $\elevel_\G(0)=0$.
\end{proof}

The following is the main result of this section.

\begin{theorem}
\label{ps}
Any minimizing sequence $\{u_n\}\subset H_\mu^1(\G)$ is weakly compact in $H^1(\G)$.
If $u_n\rightharpoonup u$ weakly in $H^1(\G)$,
then, \emph{either}
\begin{itemize}
\item[(i)] $u_n\to 0$ in $L^\infty_{\text{loc}}(\G)$ and $u\equiv 0$, \emph{or}
\item[(ii)] $u\in H^1_\mu(\G)$, $u$ is a minimizer and $u_n\to u$ \emph{strongly} in $H^1(\G)\cap L^p(\G)$.
\end{itemize}
\end{theorem}

\begin{proof} We can assume $\G$ is noncompact, otherwise (ii) is automatic due to compact embeddings.
The boundedness of $\{u_n\}$ in $H^1(\G)\cap L^p(\G)$ follows from Remark~\ref{remgen}.
Now assume that $u_n\rightharpoonup u$ in $H^1(\G)$,
so that $u_n\to u$ also in $L^\infty_{\text{loc}}(\G)$. From weak convergence in $L^2(\G)$, we have
\begin{equation}
\label{defm}
m:=\int_\G |u|^2 \,dx \le \liminf_n \int_G |u_n|^2 \,dx =\mu.
\end{equation}
%If $m = \mu$, then $u\in H_\mu^1(\G)$ and the convergence is strong in $L^2(\G)$; this entails by interpolation that the convergence is also strong in $L^p(\G)$. Then by weak lower semicontinuity,
%\[
%E(u,\G) \le \liminf_n \frac12\int_\G |u_n'|^2\,dx-\lim_n \frac1p
%\int_\G |u_n|^p\,dx \le \liminf_n E(u_n,\G) = \elevel_\G(\mu).
%\]
%Therefore
%\[
%o(1) = E(u_n,\G) - E(u,\G) = \frac12\int_\G |u_n'|^2\,dx
%-\frac12\int_\G |u'|^2\,dx + o(1),
%\]
%showing that $u'_n$ converges strongly in $L^2(\G)$ to $u'$, and contradicting the fact that $u_n$ does not converge strongly.
%
%Assume now that $0< m<\mu$.
Since $u_n \to u $ pointwise and $\|u_n\|_{L^p}$ is uniformly bounded, by the Brezis--Lieb lemma (\cite{BL})
we have
\[
\frac 1 p\int_\G |u_n|^p\,dx - \frac 1 p\int_\G |u_n-u|^p\,dx - \frac 1 p\int_\G |u|^p\,dx =
o(1)
\]
as $n\to \infty$. By weak convergence of $u_n'$ we also have
\[
\frac 1 2\int_\G |u_n'|^2\,dx - \frac 1 2\int_\G |u_n'-u'|^2\,dx -
\frac 1 2\int_\G
|u'|^2\,dx = o(1)
\]
so that, taking the difference,  we find
\begin{equation}
\label{split}
E(u_n,\G) = E(u_n-u, \G) + E(u,\G) + o(1)
\end{equation}
as $n \to \infty$. Therefore, letting $\nu_n= \int_\G |u_n-u|^2\,dx$, we have
\[
E(u_n,\G) \geq  \elevel_\G(\nu_n) + E(u,\G) + o(1), \quad\text{as $n\to\infty$.}
\]
Letting $n\to\infty$, from $u_n\overset{L^2}{\rightharpoonup} u$ we have $\nu_n\to \mu-m$ and,
since $\elevel_\G$ is continuous (Theorem~\ref{subadd}), from the previous inequality we obtain
\begin{equation}
\label{superadd}
\elevel_\G(\mu)\geq \elevel_\G(\mu-m)+E(u,\G)\geq \elevel_\G(\mu-m)+\elevel_\G(m).
\end{equation}
Since $\elevel_\G$ is strictly subadditive,
%all the inequalities in \eqref{superadd} are in fact equalities, and
if $0<m<\mu$ we would have
$\elevel_\G(\mu-m)+\elevel_\G(m)>\elevel_\G(\mu)$, which plugged into \eqref{superadd} would
give a contradiction. Therefore, either $m=0$ (and $u\equiv 0$),
or $m=\mu$, in which case $u\in H^1_\mu(\G)$. Moreover, in this case
$\elevel_\G(\mu-m)=\elevel_\G(0)=0$, and
\eqref{superadd} reveals that $\elevel_\G(\mu)\geq E(u,\G)$, whence $u$ is a minimizer.
%To complete the proof, we have to show that  $u_n\to u$ strongly in $H^1(\G)$.
Moreover,  when $m=\mu$, \eqref{defm} shows that $u_n\to u$ strongly in $L^2(\G)$, hence
also strongly in $L^p(\G)$, since $p>2$ and $\Vert u_n\Vert_{L^\infty}\leq C$
by  Remark~\ref{remgen}. But
\[
\frac 1 2 \Vert u'\Vert_{L^2}^2-
\frac 1 p\Vert u\Vert_{L^p}^p=
  E(u,\G)=\elevel_\G(\mu)=
  \lim_{n}
\left(\frac 1 2 \Vert u'_n\Vert_{L^2}^2-
\frac 1 p \Vert u_n\Vert_{L^p}^p\right),
\]
and strong convergence in $L^p$ implies the convergence of the $L^2$-norms of
$u'_n$ to the corresponding norm of $u'$. Hence, $u_n\to u$ strongly in $H^1(\G)$.
\end{proof}

\begin{theorem}
\label{main1}
Let $\G$ be a noncompact graph. If
\begin{equation}
\label{below}
\inf_{v\in H_\mu^1(\G)} E(v,\G) < \min_{\phi\in H_\mu^1(\R)}
E(\phi,\R),% = E(\phi_\mu,\R),
\end{equation}
then the infimum is attained.
\end{theorem}

\begin{proof}
Given a minimizing sequence $\{u_n\}\subset  H_\mu^1(\G)$, it suffices to rule
out case (i) of Theorem~\ref{ps}. Therefore,
we assume that $u_n\to 0$ in $L^\infty_{\text{loc}}(\G)$, and we seek a contradiction.

If $\eps_n$ denotes the maximum of $u_n$ on the compact core of $\G$, clearly
$\eps_n\to 0$.
Therefore, letting $v_n:=\max\{0, u_n-\eps_n\}$,
since $\G$ contains at least one half-line along which $u_n\to 0$,
we see that the number of preimages $v_n^{-1}(t)$ is at least $2$
for every $t\in(0, \max v_n)$ and every $n$. If $\widehat v_n$ is the symmetric rearrangement of $v_n$ on $\R$,
Proposition~3.1 of \cite{ast} shows that
\[
E(v_n,\G) \ge E(\widehat v_n,\R) \ge \min_{\phi \in H^1_\mu(\R)} E(\phi,\R).
\]
Moreover, since $\| u_n-v_n\|_{H^1(\G)}=o(1)$ as $n\to \infty$,  $\{v_n\}$ is still
a minimizing sequence so that, letting $n\to\infty$ in the previous inequality,
the strict inequality in \eqref{below} is violated.
\end{proof}

The following immediate corollary is the operative version of the preceding result.

\begin{corollary}
\label{oper}
Let $\G$ be a noncompact graph. If there exists $u\in H_\mu^1(\G)$ such that
\begin{equation}
\label{assoper}
E(u,\G) \le  E(\phi_\mu,\R),
\end{equation}
then $\G$ admits a ground state of mass $\mu$.
\end{corollary}

\begin{proof} If the mentioned $u$ is not a ground state, then \eqref{below} is satisfied because
$\phi_\mu$ is a ground state on $\R$,
and Theorem~\ref{main1} applies.
\end{proof}

We now show, through a series of examples, how Corollary~\ref{oper} can be applied,
%in some concrete cases,
to detect the existence of ground states.

In all the examples below, a soliton of arbitrary mass is
cut, and its pieces are pasted and possibly partly rearranged on the graph $\G$ under study.
The construction always ends up with a function in the space $H^1(\G)$
with the same mass as the original soliton, but with a lower (or equal) energy: then,
as a result of Corollary~\ref{oper}, all the graphs in Figures \ref{bubbles}--\ref{3-fork}
admit a ground state of mass $\mu$, for every $\mu>0$.

\medskip

First, in Fig. \ref{bubbles} it is shown how a soliton, due to its symmetries, can be placed on
a {\em line with two circles}, i.e. on a graph $\G$ made up of two half-lines,
two edges each of length $\ell_1/2$, and a self-loop of length $\ell_2$, arranged
as in Figure~\ref{bubbles}.b (the resulting function achieves its maximum at the north pole
of the upper circle).
This procedure does not affect the
value of the energy, so a ground
state exists by Corollary \ref{oper} (in fact, as shown in %Example 2.4 and Theorem 2.5 in
\cite{ast}, the ground state is the function just now constructed).

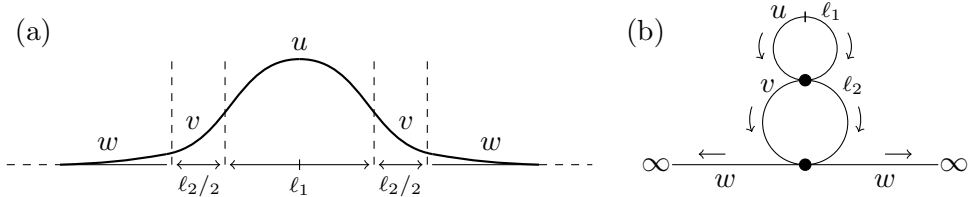
\begin{figure}[ht]    % solitone e bolle
\begin{center}
%soliton
\begin{tikzpicture}[xscale= 0.7,yscale=0.7]
\draw[-,thick] (-6,2) to [out=0,in=170] (-3.5,0.2);
\draw[-,thick] (-6,2) to [out=180,in=10] (-8.5,0.2);
\draw[-,thick] (-3.5,0.2) to [out=350,in=178] (-1.5,0);
\draw[-,thick] (-8.5,0.2) to [out=190,in=2] (-10.5,0);
\node at (-8.4,0)  [infinito] (6) {\phantom{.}};
\draw [dashed] (-8.4,0) -- (-8.4,2);
\node at (-7.4,0)  [infinito] (7) {\phantom{.}};
\draw[dashed] (-7.4,0)--(-7.4,2);
\node at (-4.6,0)  [infinito] (8) {\phantom{.}};
\draw[dashed] (-4.6,0)--(-4.6,2);
\node at (-3.6,0)  [infinito] (9) {\phantom{.}};
\draw[dashed] (-3.6,0)--(-3.6,2);
\draw[<->] (6) -- (7);
\draw[<->] (7) -- (8);
\draw[<->] (8) -- (9);
\draw[-] (9) -- (-1.5,0);
\draw[-] (6) -- (-10.5,0);
\draw[-] (-6,-0.1) -- (-6,0.1);
\draw[dashed] (-11.5,0) -- (-10.5,0);
\draw[dashed] (-1.5,0) -- (-.5,0);
\node at (-6,2.3)  {$u$};
\node at (-4,.8)  {$v$};
\node at (-8,.8)  {$v$};
\node at (-2.4,.4)  {$w$};
\node at (-9.6,.4)  {$w$};
\node at (-6,-0.4) {$\scriptstyle{\ell_1}$};
\node at (-4.1,-0.4) {$\scriptstyle{\ell_2/ _2}$};
\node at (-7.88,-0.4) {$\scriptstyle{\ell_2/ _2}$};
\node at (-11,2.5) {(a)};
% bolle
\draw[-] (1,0) -- (6,0);
\node at (.7,0) [infinito]  {$\infty$};
\node at (6.3,0) [infinito]  {$\infty$};
\node at (3.5,0) [nodo] (1) {};
\node at (3.5,1.6) [nodo] (2) {};
\draw(3.5,.8) circle (.8);
\draw(3.5,2.2) circle (.6);
%arrows
\draw [->] (2.5,1.1) arc [radius=.7, start angle=160, end angle= 200];
\draw [->] (4.5,1.1) arc [radius=.7, start angle=20, end angle= -20];
\draw [->] (2.7,2.5) arc [radius=.5, start angle=150, end angle= 210];
\draw [->] (4.3,2.5) arc [radius=.5, start angle=30, end angle= -30];
\draw[->] (2,0.2) -- (1.5,0.2);
\draw[->] (5,0.2) -- (5.5,0.2);
\draw[-] (3.5,2.7) -- (3.5,2.9);
\node at (4,2.9) {$\scriptstyle{\ell_1}$};
\node at (4.4,1.5) {$\scriptstyle{\ell_2}$};
\node at (3,2.9) {$u$};
\node at (2.8,1.5) {$v$};
\node at (2,-.3) {$w$};
\node at (5,-.3) {$w$};
\node at (0.5,2.5) {(b)};
\draw [-] (3.5,2.7) -- (3.5,2.9);  % taglietto verticale sul polo nord
\end{tikzpicture}
\end{center}
\caption{\footnotesize{{\bf a line with two circles}. A soliton $\phi_\mu$, in (a), can be
folded and placed, without changing its energy, on the graph in (b): the two circles have length
$\ell_1$ and $\ell_2$, while arrows indicate the directions in which the function
decreases.}}
\label{bubbles}
\end{figure}

\medskip As a second example we consider a {\em signpost} graph $\G$,
made up of two half-lines, one edge of length $\ell_2$, and another edge of length $\ell_1$
that forms a self-loop, as
in Fig. \ref{palina}. The caption
explains how a function can be constructed, with an energy level strictly
lower than the soliton's, starting from the function in Figure~\ref{bubbles}.b.

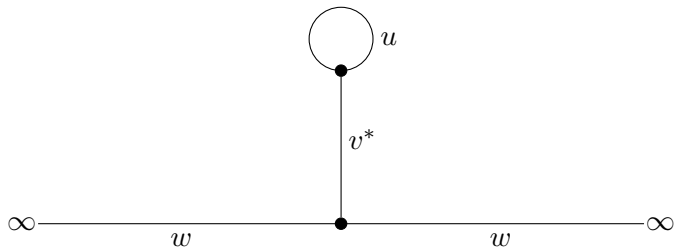
\begin{figure}[ht]% palina
\begin{center}
\begin{tikzpicture}[xscale= 0.7,yscale=0.7]
\node at (6,0) [infinito]  (1) {$\infty$};
\node at (12,0) [nodo] (2) {};
\node at (18,0) [infinito]  (3) {$\infty$};
\node at (16,0)  [minimum size=0pt] (5) {};
\node at (8,0) [minimum size=0pt] (4) {};
\node at (12,2.9) [nodo] (6) {};
\draw[-] (1) -- (3);
\draw(12,3.5) circle (0.6);
\draw[-] (2) -- (6);
\node at (12.9,3.5)  {$u$};
\node at (12.4,1.6)  {$v^*$};
\node at (9,-.3)  {$w$};
\node at (15,-.3)  {$w$};
\end{tikzpicture}
\end{center}
\caption{\footnotesize{{\bf a signpost graph.}
The folded soliton  of Fig. \ref{bubbles}.b can be placed on the
signpost, after a monotone rearrangement of $v$ from the circle to an interval $I$ of length $\ell_2$,
thus obtaining
a function $v^* \in H^1 (I)$. Energy is decreased by the rearrangement.
 }}
\label{palina}
\end{figure}

\noindent
A similar strategy works when $\G$ consists of  two half-lines and a terminal edge,
as explained in Figure~\ref{baffo}. The existence of ground states for this $\G$
was established in \cite{ast} by ad hoc techniques, specific to this graph.

\begin{figure}[ht] \label{baffo} % retta baffo
\begin{center}
\begin{tikzpicture}[xscale= 0.7,yscale=0.7]
\node at (6,0) [infinito]  (1) {$\infty$};
\node at (12,0) [nodo] (2) {};
\node at (18,0) [infinito]  (3) {$\infty$};
\node at (16,0)  [minimum size=0pt] (5) {};
\node at (8,0) [minimum size=0pt] (4) {};
\node at (12,1.9) [nodo] (6) {};
\node at (12,3.3) [nodo] (7) {};
\draw[-] (1) -- (3);
%\draw[dashed] (4) -- (1);
%\draw [-] (4) -- (2);
\draw [-] (6) -- (7);
%\draw [-] (5) -- (2) ;
%\draw[dashed] (5) -- (3);
\draw[-] (2) -- (6);
\node at (12.4,2.8)  {$u$};
\node at (12.4,1)  {$v^*$};
\node at (9,-.3)  {$w$};
\node at (15,-.3)  {$w$};
\end{tikzpicture}
\end{center}
{\caption {\footnotesize{\bf line with a terminal edge.}
%The energy of the soliton can be lowered by rearranging $v$ exactly
%like in the signpost case. In addition, the self-loop of length $\ell_1$
%in Fig. \ref{bubbles} is unfolded into a terminal edge.
Starting from the function in Figure~\ref{palina}, whose energy level
is lower than the soliton's,
 one may further unfold
the self-loop, obtaining a function fitted to a line with a single pendant, of length $\ell_1\!+\!\ell_2$.
}}
\end{figure}
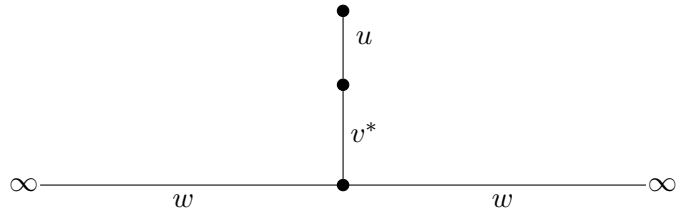

\medskip

\noindent Also for a {\em tadpole} graph, i.e. a half-line attached to a self-loop as in Figure~\ref{tadpole},
one can easily construct a function more performant than the soliton, as explained in the caption.
Tadpole graphs were considered in (\cite{noja,pelinovsky}), as concerns the
existence of solutions to \eqref{euleroforte} with the Kirchhoff condition: now we know that,
among those solutions, there exist ground states of arbitrary prescribed mass.

\begin{figure}[ht]
\begin{center}
\begin{tikzpicture}[xscale= 0.7,yscale=0.7]
\node at (-3,0) [infinito]  (1) {$\infty$};
\node at (3,0) [nodo]  (3) {};
\node at (-2,0) [minimum size=0pt] (4) {};
\draw [-] (1,-.1) -- (1,.1);
\draw [-] (1) -- (3);
\node at (4.2,.7) {$u$};
\node at (2,-.3) {$v^*$};
\node at (-.7,-.3) {$w^*$};
\draw(3.6,0) circle (.6);
\end{tikzpicture}
\end{center}
\caption{\footnotesize{{\bf a tadpole graph.}
The energy of the soliton can be lowered by rearranging $v$ as in
Fig. \ref{palina}.
%Furthermore, the two soliton tails
%denoted by $w$ in  Fig. \ref{bubbles}.b are rearranged
%monotonically together into the function $w^* \in H^1 (0, + \infty)$.
Furthermore, by a monotone rearrangement from $H^1(\R)$ to $H^1(\R^+)$, the two soliton tails
denoted by $w$ can be melted into a single function $w^*$.
}}
\label{tadpole}
\end{figure}
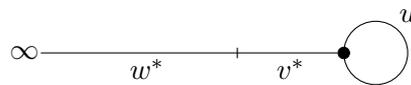

\medskip

\noindent As a last example, we consider a $3$-{\em fork} graph, which can be handled as explained in Fig.~\ref{3-fork}.
By a proper choice of the initial lengths $\ell_1,\ell_2$, and of the point where the second self-loop is
opened, the lengths of the three spikes can be made arbitrary. The case of a \emph{2-fork} can be treated as
a degenerate case, when $\ell_2=0$ in Figure~\ref{bubbles}.a, so that no $v$ is needed and
the lower circle in Figure~\ref{bubbles}.b disappears.

\medskip

Finally, we mention that several other examples can be made if, for instance, one
folds a soliton and fits it to
a line
with a ``tower of $n$ circles''  ($n\geq 2$) stacked together: Figure~\ref{bubbles} corresponds
to $n=2$, but the construction is easily generalized when $n>2$.

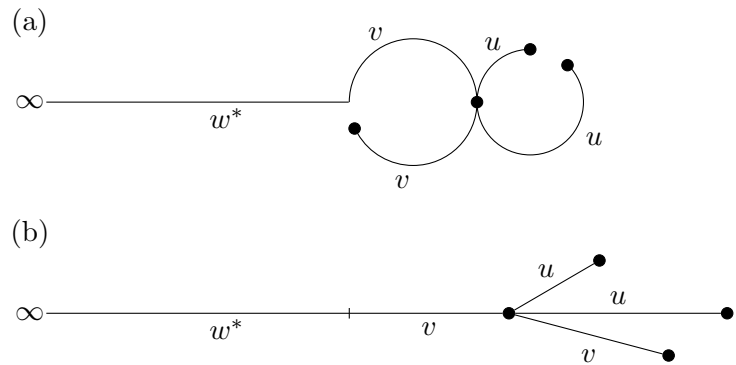
\begin{figure}[ht] \label{bolleaperte}
\begin{center}
\begin{tikzpicture}[xscale= 0.7,yscale=0.7]
%\foreach \yy in {1,3}
%{
%\draw  (2,\yy) -- (3,\yy);
%};
%\draw (2,1) -- (3,1);
\node at (2.4,0) [nodo] {};
\node at (3.4,1) [nodo] {};
\node at (4.1,.7) [nodo] {};
\node at (.1,-.5) [nodo] {};
\node at (0,0) [infinito] (0) {};
\node at (-6,0) [infinito] (1) {$\infty$};
\node at (-6,1.5) {(a)};
\draw [-] (0) -- (1);
\draw [-] (0,0) arc [radius=1.2, start angle=180, end angle= -150];
\draw [-] (2.4,0) arc [radius=1, start angle=180, end angle=90];
\draw [-] (2.4,0) arc [radius=1, start angle=-180, end angle= 45];
\node at (-2.3,-.3) {$w^*$};
\node at (.5,1.3) {$v$};
\node at (2.7,1.1) {$u$};
\node at (4.6,-.7) {$u$};
\node at (1,-1.5) {$v$};
%\end{tikzpicture}
%
%\begin{tikzpicture}
%% tre baffi
\node at (-6,-4) [infinito]  (1) {$\infty$};
\node at (-2.3,-4.3) {$w^*$};
\coordinate (2bis) at (0,-4);
\draw [-] [label=above:ciao] (1) -- (2bis);
\node at (-6,-2.5)  {(b)};
\draw [-] (0,-4.1) -- (0,-3.9);

\node at (1.5,-4.3) {$v$};

\begin{scope}[shift={(3,0)}]
\node at (0,-4) [nodo] (2) {};
\draw [-] (2bis) -- (2);
\node at (4.1,-4) [nodo]  (3) {};
\node at (-3.5,-4) [minimum size=0pt] (6) {};
\node at (1.7,-3) [nodo] (7) {};
\node at (3,-4.8) [nodo] (8) {};
%\draw [-] (2) -- (6) ;
\draw [-] (3) -- (2) ;
\draw [-] (2) -- (7) ;
\draw [-] (2) -- (8) ;
\node at (.7,-3.2) {$u$};
\node at (2.05,-3.7) {$u$};
\node at (1.5,-4.8) {$v$};
\end{scope}
\end{tikzpicture}
\end{center}
 \caption{\footnotesize{{\bf a 3-fork graph.} Construction of a competitor:
first, by a monotone rearrangement, the two tails $w$ in Figure~\ref{bubbles}.b are
melted into $w^* \!\in\! H^1 (\R^+)$, then the two self-loops are opened  (a).
The resulting function now fits the $3$-fork graph (b).
}}
\label{3-fork}
\end{figure}

\section{Stability and threshold phenomena}

The main purpose of this section is to show that the existence of
a ground state on a graph does not depend on the topology of the
graph only, but also on the interplay between its metric properties (e.g the length of some edges) and the
mass $\mu$. In some cases this will lead to sharp threshold phenomena, in which a graph
$\G$ admits a ground state {\em if and only if} a certain quantity exceeds a critical value.

\noindent Since, in order to have existence, assumption (H) must be
violated, we focus on the simplest possible violation of it,
namely the presence of a terminal edge, that we shall model as the
interval $[0, \ell]$. The first result is a consequence of Corollary \ref{oper} and shows that if
$\mu^\beta \ell $ is large enough then a minimizer exists.

\begin{proposition}
\label{corbaffolungo}
Let $\G$ be a noncompact graph with a terminal edge of length $\ell$.
There exists $C_p^* > 0$ such that if $\mu^\beta\ell \geq  C_p^*$,
then $\G$ admits a ground state.
\end{proposition}

\begin{proof}
By scaling (Remark~\ref{remscal}), we may assume that $\mu = 1$. Since
$2^{2\beta}>1$, putting $\mu=1$ in \eqref{energiamezzosol} and \eqref{energiasol} we
have $E(\phi_2,\R^+) < E(\phi_1, \R)$: therefore,
if $C^*_p$ is large enough,
there exists a function $u \in H^1(\R^+)$
such that
\[
E(u,\R^+) \leq E(\phi_1, \R),\quad \int_0^{+\infty} |u|^2\,dx=1,\quad u(x)=0\quad\forall x\geq C_p^*
\]
(for instance, one may take $(\phi_2-\eps)^+$ with $\eps$ small, and then renormalize it in
$L^2(\R^+)$). Now, if $\ell\geq C_p^*$,
we may interpret $[0,\ell]$ as the terminal edge of $\G$
(attached to $\G$ at $x=\ell$)
and extend $u\equiv 0$ on the rest of $\G$: we have thus constructed a function $u\in H^1(\G)$
of mass $1$, such that
$E(u,\G) < E(\phi_1, \R)$. By Corollary \ref{oper}, $\G$ admits a ground state of mass $1$.
\end{proof}

The next result illustrates a sort of stability as regards the existence of ground states. In particular
it says, loosely speaking, that the ``limit'' of shrinking graphs carrying a ground state also carries a ground state.

Let $\G$ be any noncompact graph and, for every $n\in {\mathbb
N}$, let $\K_n$ be a connected compact graph, of \emph{total length}
$|\K_n|$. We denote by
$\G_n$ the graph obtained by attaching $\K_n$ to $\G$, at some (fixed)
point $\vv$ of $\G$. Notice that, in this way,
$\G$ can be regarded as a subgraph of $\G_n$.

\begin{remark}\label{rembaf} The simplest example of this structure is when each $\K_n$ is a single edge
of length $\ell_n$, attached
to $\G$ as a (further) terminal edge.
\end{remark}

\begin{theorem}
\label{stab} Let $\mu >0$. If every $\G_n$
admits a ground state $u_n$ of mass $\mu$, and $|\K_n| \to 0$ as $n\to
\infty$, then also $\G$ admits a ground state of mass $\mu$.
\end{theorem}

\begin{proof}
Arguing by contradiction, let us assume that $\G$ has no ground state of mass 1
(by scaling, we may assume $\mu=1$). Then from \eqref{trapped} and \eqref{energiasol} we have
\begin{equation}
\label{pre1}
\elevel_{\G_n}(1)=E(u_n,\G_n)\leq -\theta_p,\qquad\elevel_\G(1)=-\theta_p,
\end{equation}
otherwise a ground state would exist on $\G$ too, by Theorem~\ref{main1}.
Letting $\sigma_n:=\int_{\G}|u_n|^2 dx$, by restriction to $\G$ we define $v_n:=\sigma_n^{-1/2} u_n$,
as a function in $H^1(\G)$ of mass 1. Since $v_n$ is not a ground state, $\sigma_n<\mu=1$ and $p>2$, we have from
\eqref{pre1} and \eqref{NLSe}
\begin{equation}
\begin{split}
\label{pre2}
-\theta_p<E(v_n,\G)\leq \frac 1 {\sigma_n} E(u_n,\G)
=\frac 1 {\sigma_n} \bigl( E(u_n,\G_n)-E(u_n,\K_n)\bigr)\\
=\frac 1 {\sigma_n} \bigl( -\theta_p-E(u_n,\K_n)\bigr)
\leq \frac 1 {\sigma_n} \left( -\theta_p+ \frac 1 p\int_{\K_n} |u_n|^pdx\right).
\end{split}
\end{equation}
Now, since $|\K_n|\to 0$ while $|u_n|\leq C_p$ by \eqref{st30}, we have
\[
\lim_{n\to\infty} \int_{\K_n} |u_n|^2\,dx
=0,\qquad
\lim_{n\to\infty} \int_{\K_n} |u_n|^p\,dx=0,
\]
and the first limit entails that $\sigma_n\to 1$, since $\Vert u_n\Vert_{L^2(\G_n)}^2=\mu=1$.
As a consequence, letting $n\to\infty$ in \eqref{pre2}, we have $E(v_n,\G)\to-\theta_p$ so that
$\{v_n\}$ is a minimizing sequence by \eqref{pre1}. Since $\G$ has no ground state, case (i) of Theorem~\ref{ps}
must occur relative to $\{v_n\}$, so that $v_n\to 0$ in $L^\infty_{\text{loc}}(\G)$ and hence, since
$\sigma_n\to 1$, also $u_n\to 0$ in $L^\infty_{\text{loc}}(\G)$. In particular,
if $\vv$ is the
point of $\G$ where each $\K_n$ is attached,
$u_n(\vv)\to 0$. Since $\K_n$ is connected,
\[
\max_{\K_n} u_n\leq u_n(\vv)+\int_{\K_n} |u'|\,dx \leq
u_n(\vv)+|\K_n|^{\frac 1 2}\cdot\Vert u_n'\Vert_{L^2(\G_n)}
\]
and, since $|\K_n|\to 0$,  applying \eqref{st29} to $u'_n$ on $\G_n$ we see that
\begin{equation}
\label{pre3}
M_n:=\Vert u_n\Vert_{L^\infty(\K_n)}\to 0\quad\text{as $n\to\infty$.}
\end{equation}
Now, multiplying \eqref{pre2} by $\sigma_n$, moving $-\theta_p$ to the left hand side,
and using $M_n$ to partially estimate the last integral, we find
\[
(1-\sigma_n) \theta_p<\frac{M_n^{p-2}}p \int_{\K_n} \!\!|u_n|^2dx
=
\frac{M_n^{p-2}}p \left(\int_{\G_n} \!\!|u_n|^2dx-\sigma_n\right)
=
\frac{M_n^{p-2}}p (1-\sigma_n)
\]
having used $\mu=1$ in the last passage. Since $\sigma_n\in (0,1)$, dividing by $1-\sigma_n$ and letting
$n\to\infty$, from \eqref{pre3} we obtain $\theta_p\leq 0$, which is a contradiction since $\theta_p>0$ by \eqref{energiasol}.
\end{proof}

We are now ready to describe the threshold phenomenon mentioned in the Introduction.
We will show that a sharp phase transition between existence and nonexistence of ground sates may occur,
when a certain quantity  passes through a critical value.

To illustrate this phenomenon, we consider the case where $\G$
is a star--graph made up of several half-lines and one bounded edge of length $\ell$:
existence of ground states of mass $\mu$ then depends on the product $\mu^\beta \ell$,
and a precise threshold arises.
Our techniques are to some extent ad hoc,
and it is an open problem to investigate the validity of this threshold phenomenon,
on a generic graph $\G$.

\begin{theorem}\label{teoksemirette}
Let $\G_\ell$ denote the graph made up of $N$ half-lines ($N\geq 3$)
and a terminal edge of length $\ell$, all emanating from the same vertex $\vv$.
Then there exists $C^*>0$ (depending only on $N$ and $p$) such that
\begin{equation}
 \inf_{w\in
H^1_\mu(\G_\ell)}E(w,\G_\ell)\text{ is attained} \quad\iff\quad \mu^\beta \ell \ge C^*.
\end{equation}
\end{theorem}

\begin{proof}
By scaling, we may assume that $\mu=1$, and define $C^*$
as the infimum of those $\ell>0$, such that $\G_\ell$ admits a ground state
of mass $1$ (the finiteness of $C^*$ follows from Proposition~\ref{corbaffolungo}).
We have $C^*>0$, otherwise, if $\G_\ell$ had a ground state
for arbitrarily small $\ell$, then by Theorem~\ref{stab} (applied
with $\K_n$ as in Remark~\ref{rembaf}) $\G_0$ would have a ground state too:
but since $\G_0$ satisfies Assumption (H),
by Theorem~2.5 in \cite{ast} it admits no ground state.
Moreover, considering a sequence $\ell_n\downarrow C^*$ such that $\G_{\ell_n}$
carries a ground state, we see from Theorem~\ref{stab} that $\G_\ell$ has a ground state,
when $\ell=C^*$.

To complete the proof we now show that, if $\ell'>\ell$ and $\G_\ell$ has a ground state
$u$ of mass $1$, then $\G_{\ell'}$ has a ground state too. Relying on Corollary~\ref{oper},
it is sufficient to construct a competitor $v\in H^1(\G_{\ell'})$, of mass $1$, such that
\begin{equation}
\label{compv} E(v,\G_{\ell'})\leq E(u,\G_\ell).
\end{equation}
Our competitor will be constructed starting from
the ground state $u$, and fitting it to $\G_{\ell'}$ by
the following geometric surgery procedure on $\G_\ell$.

By Theorem~\ref{teocode}, the restriction of $u$ to \emph{any} half-line of
$\G_\ell$ is always the \emph{same} function $\phi:[0,+\infty)\to\R$.
One begins by removing from every half-line of $\G_\ell$
its initial interval $[0,\delta)$, where $\delta=(\ell'-\ell)/N$. The remaining (unbounded) portions of these $N$
half-lines, glued at their starting point, form the $N$ half-lines of $\G_{\ell'}$: on this part
of $\G_{\ell'}$ (to which we shall later attach a pendant of length $\ell'$), the competitor $v$ is defined in a natural way, as the
restriction of the original $u$.
We are thus left with
$N$ copies of the interval $[0,\delta)$  %(one from each original half-line of $G_\ell$,
(each carrying a copy of $\phi:[0,\delta)\to\R$)
and the original pendant of $\G_\ell$ (an interval $I_\ell$ of length $\ell$ that carries a portion of
$u$): these can be used to construct the pendant of $\G_{\ell'}$ (and define $v$ there), as follows.
First, on the interval $I=[0,N\delta]=[0,\ell'-\ell]$, we define
\[
v\in H^1(I),\quad
v(x):=\phi(\delta-x/N),\quad
0\leq x\leq N\delta,
\]
that is, a \emph{horizontal stretching} of $\phi_{[0,\delta)}$ by a factor $N$, combined with a
reflection. It is clear that,
on passing from the $N$ copies of $\phi_{[0,\delta)}$, to $v$ over $I$, the overall $L^2$ and
$L^p$ norms (in fact, all the $L^r$ norms) are preserved, while the kinetic part of the energy
is reduced by a factor $1/N$ (which accounts for the inequality in \eqref{compv}).
Then, since by construction $v(0)=\phi(\delta)$, the interval $I$ can be attached
to the $N$
half-lines previously constructed, in such a way that $v$ is continuous. Finally,
the original pendant $I_\ell$ of $\G_\ell$ can be further attached to $I$ (forming
one single pendant of length $\ell'$) and, since $v(N\delta)=\phi(0)=u(\vv)$,
the original $u$ on $I_\ell$ can be used to extend $v$, from $I$ to $[0,\ell']$, in a continuous fashion.
Thus, by construction, the total mass of $v$ equals that of $u$, and \eqref{compv} is satisfied.
\end{proof}

\section{Graphs with one half-line and no ground state}
\label{sectionscopino} In this section we consider noncompact
metric graphs $\G$ consisting of a compact core $\K$ with the
addition of \emph{one} half-line attached to it, and we show that
ground states with a prescribed mass may fail to exist. As
mentioned in the Introduction, this issue is nontrivial since
graphs with just one half-line cannot satisfy assumption~(H) of
\cite{ast} (which would automatically rule out ground states of
any mass): on the contrary, their structure rather seems to favour
a fruitful use of Corollary~\ref{oper}, and this makes the
construction of counterexamples quite involved.

If $\K$ is \emph{compact} metric graph,
 the following Sobolev inequality is well known
(see e.g. Corollary 2.2 in \cite{haeseler})
\begin{equation}
\label{poincare} \Vert u\Vert_{L^\infty(\K)} \leq |\K|^{-\frac 1
2} \Vert u\Vert_{L^2(\K)} + \diam(\K)^{\frac 1 2} \Vert
u'\Vert_{L^2(\K)},\quad\forall u\in H^1(\K),
\end{equation}
where $|\K|$ is the total length of $\K$ while $\diam(\K)$ is its
diameter. It turns out that, for fixed mass, a big length and a
small diameter may represent an obstruction to the existence of
ground states.

\begin{theorem}\label{teoscopino}
There exists a number $\eps>0$, depending only on $p$, with the
following property. If $\mu>0$ and $\G$ consists of one half-line
and a compact core $\K$ satisfying
\begin{equation}
\label{assKpoinc} \max\left\{ \mu^\beta \diam(\K),\frac 1
{\mu^\beta |\K|}\right\}<\eps,
\end{equation}
then $\G$ has no ground state of mass $\mu$.
\end{theorem}
\begin{proof}
Assuming $\eps$ small enough and the existence of a ground state
$u\in H^1_\mu(\G)$, we shall prove that
\begin{equation}
\label{contrad} E(u,\G)-E(\phi_\mu,\R)
> 0,
\end{equation}
which is a contradiction since it violates the second inequality
in \eqref{trapped2}.

Choosing a coordinate $x\in [0,\infty)$ on the half-line
$\G\setminus\K$, by Corollary~\ref{corlambda} the restriction of
$u$ to $\G\setminus\K$ takes the form
\[
 \phi_m(x+y),\quad x\geq 0,
\]
for some $y\in \R$ and $m>0$ satisfying \eqref{stimaMuniv}.

\medskip

\noindent{\em (i) Proof that $y<0$, and $L^\infty$ estimate. }
%We first prove that $y<0$.
Combining \eqref{poincare} with \eqref{assKpoinc} yields
\begin{equation}
\label{poinc2} \Vert u\Vert_{L^\infty(\K)}^2 \leq
2\eps\left(\mu^{\beta}\int_\K u^2\,dx + \mu^{-\beta} \int_\K
|u'|^2\,dx\right)
\end{equation}
and, since $\Vert
u\Vert_{L^2(\K)}^2\leq %u\Vert_{L^2(\G)}^2=
\mu$ and $\Vert u'\Vert_{L^2(\K)}^2\leq \Vert u'\Vert_{L^2(\G)}^2$,
using \eqref{st29} we obtain
\begin{equation}
\label{st31} \Vert u\Vert_{L^\infty(\K)}^2 \leq \eps
C\mu^{\beta+1}.
\end{equation}
Now, if $y$ were positive, $\phi_m(x\!+\!y)$ would be decreasing
for $x\!\geq\! 0$ and, since the half-line of $\G$ is attached to
$\K$ at the point $x=0$, we would have $\Vert
u\Vert_{L^\infty(\G)}^2\!=\!\Vert u\Vert_{L^\infty(\K)}^2$. This,
in view of \eqref{st31}, would violate the first inequality in
\eqref{st30} for small $\eps$, hence we conclude that $y<0$.

\medskip

\noindent{\em (ii) Estimate for $\mu-m$. } Since $\Vert
\phi_m\Vert_{L^2(\R)}^2=m$, from the scaling rule
\eqref{scalingrule} written when $\mu=m$, since $2\alpha-\beta=1$
we find by a change of variable
\begin{equation}
\label{normadue} \int_{\G\setminus\K} |u|^2\,dx
%= \int_y^\infty |\phi_m|^2\,dx
=m-\int_{-\infty}^y |\phi_m|^2\,dx =m-m\int_{-\infty}^{z}
|\phi_1|^2\,dx,\quad z:=m^\beta y.
\end{equation}
As $z<0$, to estimate the last integral we introduce the auxiliary
function
\begin{equation*}
f(x):=\frac {\int_{-\infty}^x \left|
\phi_1(s)\right|^2\,ds}{\phi_1(x)^2}\quad (x\leq 0),
\end{equation*}
and observe that,  from l'H\^{o}pital's rule,
\[
\lim_{x\to-\infty} f(x)= \lim_{x\to-\infty}
\frac{\phi_1(x)^2}{2\phi_1(x)\phi_1'(x)}= \lim_{x\to-\infty}
\frac{\phi_1'(x)}{2\phi_1(x)}\in (0,+\infty)
\]
(the finiteness of the last limit follows easily from the explicit
form of $\phi_1$, discussed in Remark~\ref{remsol}). Since $f$ is
continuous on $(-\infty,0]$, its supremum $C$ is then finite, and
depends only on $p$. In particular $f(z)\leq C$, that is
\[
\int_{-\infty}^{z} |\phi_1|^2\,dx\leq C\phi_1(z)^2=C
\phi_1(m^\beta y)^2=C m^{-2\alpha}\phi_m(y)^2.
\]
Plugging into \eqref{normadue}, since $1-2\alpha=-\beta$ we obtain
\[
\int_{\G\setminus\K} |u|^2\,dx \geq
 m-C m^{-\beta} \phi_m(y)^2
\geq
 m-C \mu^{-\beta} \phi_m(y)^2
\]
having used \eqref{stimaMuniv}. Therefore, since $\mu=\int_\K
|u|^2\,dx+ \int_{\G\setminus\K} |u|^2\,dx$, we find
\begin{equation}
\label{mumenoM} \mu-m \geq \int_\K |u|^2\,dx-C \mu^{-\beta}
\phi_m(y)^2.
\end{equation}

\noindent{\em (iii) Energy estimate on the half-line. } As before,
using \eqref{scalingrule} when $\mu=m$ and letting $z=m^\beta y$,
by a change of variable we have
\begin{equation}\label{st34}
\begin{split}
 E(u,\G\setminus\K)=\int_y^\infty {\textstyle\left(\frac{|\phi_m'|^2}{2}-\frac{|\phi_m|^p}{p}\right)dx}
=m^{2\beta+1}\int_{z}^\infty {\textstyle\left(\frac{|\phi_1'|^2}{2}-\frac{|\phi_1|^p}{p}\right)dx}\\
\geq
m^{2\beta+1}\left(E(\phi_1,\R)%\int_{\R} {\textstyle\left(\frac{|\phi_1'|^2}{2}-\frac{|\phi_1|^p}{p}\right)dx}
-\int_{-\infty}^{z} {\textstyle\frac{|\phi_1'|^2}{2}dx}\right) =
-m^{2\beta+1}\left(\theta_p +\int_{-\infty}^{z}
{\textstyle\frac{|\phi_1'|^2}{2}} dx\right)
\end{split}
\end{equation}
where $\theta_p$ is as in \eqref{energiasol}. Arguing as in (ii),
now applying l'H\^{o}pital's rule to the function
$f(x)=\int_{-\infty}^x |\phi'_1|^2\,ds/\phi_1(x)^2$, we have
$f(x)\leq C_p$ for $x\leq 0$, so that
\[
m^{2\beta+1}\int_{-\infty}^{z} {\textstyle\frac{|\phi_1'|^2}{2}dx}
%=m^{2\beta+1}f(z)\phi_1(z)^2
\leq C  m^{2\beta+1}\phi_1(z)^2=C  m^{\beta}\phi_m(y)^2 \leq C
\mu^{\beta}\phi_m(y)^2
\]
having used \eqref{stimaMuniv}. Plugging into \eqref{st34} and
using \eqref{energiasol}, one finds
\begin{equation}
\label{stima18} E(u,\G\setminus\K)-E(\phi_\mu,\R)\geq \theta_p
\left(\mu^{2\beta+1}-m^{2\beta+1}\right)- C
\mu^{\beta}\phi_m(y)^2.
\end{equation}
Now observe that the mean value theorem and \eqref{stimaMuniv}
give
\begin{equation*}
C^{-1}\mu^{2\beta}\leq \frac
{\mu^{2\beta+1}-m^{2\beta+1}}{\mu-m}\leq C\mu^{2\beta}.
\end{equation*}
According to the sign of $\mu-m$, we can use the first or the
second inequality, finding in either case
\[
\mu^{2\beta+1}-m^{2\beta+1}\geq C^{-1}(\mu-m),
\]
which can be combined with \eqref{mumenoM} to obtain
\[
\mu^{2\beta+1}-m^{2\beta+1}\geq C^{-1}\mu^{2\beta}\int_\K
|u|^2\,dx-C\mu^{\beta} \phi_m(y)^2.
\]
Finally, plugging into \eqref{stima18} we have
\begin{equation}
\label{stima19} E(u,\G\setminus\K)-E(\phi_\mu,\R)\geq
C^{-1}\mu^{2\beta}\int_\K |u|^2\,dx-C\mu^{\beta} \phi_m(y)^2.
\end{equation}

\noindent{\em (iv) Energy estimate on $\K$. } We have from
\eqref{st31}
\[
\int_{\K} |u|^p\,dx\leq \Vert u\Vert_{L^\infty(\K)}^{p-2}
\int_{\K} |u|^2\,dx\leq C\eps^{\frac{p-2}{2}}\mu^{2\beta}\int_{\K}
|u|^2\,dx,
\]
and therefore
\[
E(u,\K)\geq \frac 1 2\int_\K
|u'|^2\,dx-C\eps^{\frac{p-2}{2}}\mu^{2\beta}\int_{\K} |u|^2\,dx.
\]
Summing this inequality with \eqref{stima19}, and observing that
$\phi_m(y)\leq\Vert u\Vert_{L^\infty(\K)}$, for small $\eps$ we
obtain
\[
\begin{split}
E(u,\G)-E(\phi_\mu,\R) \geq
\left(C^{-1}-C\eps^{\frac{p-2}{2}}\right)\mu^{2\beta}\int_{\K}
|u|^2\,dx + \frac 1 2\int_\K |u'|^2\,dx \\-C\mu^{\beta}
\phi_m(y)^2 \geq \frac{C^{-1}}2 \mu^{2\beta}\int_{\K} |u|^2\,dx +
\frac 1 2\int_\K |u'|^2\,dx -C\mu^{\beta} \Vert
u\Vert_{L^\infty(\K)}^2.
\end{split}
\]
Finally, factoring $\mu^\beta$ in the last expression, and using
\eqref{poinc2} to estimate $\Vert u\Vert_{L^\infty(\K)}$, we see
that \eqref{contrad} is established, if $\eps$ is small enough.
\end{proof}
It is clear from the proof that the sole scope of
assumption~\eqref{assKpoinc} is to guarantee the validity of
\eqref{poinc2} via \eqref{poincare}, and the result is still valid
if  one merely assumes \eqref{poinc2}. Assumption
\eqref{assKpoinc}, however, is handy to construct concrete
examples, as we now show.
\begin{corollary}\label{corscopino}
Let $\G$ be a star-graph made up of one half-line and $n$ edges of
length $\ell$, and let $\mu>0$. If $\mu^\beta\ell$ is small enough
(depending only on $p$), and $n$ is large enough (depending on $p$
and $\mu^\beta\ell$), then $\G$ has no ground state of mass $\mu$.
\end{corollary}
\begin{proof}
The compact core $\K$ satisfies $\diam(\K)=2\ell$
and $|\K|=n\ell$, and the application of Theorem~\ref{teoscopino}
is immediate.
\end{proof}
\begin{remark}
When $p=4$, computations with explicit constants reveal that, in
Corollary~\ref{corscopino}, $n=5$ is sufficient. Note that $n>3$
is necessary, since a $3$-fork graph as in Figure~\ref{3-fork} has
ground states for every $\mu$. It is not known, however, if a
$4$-fork graph may furnish a counterexample.
\end{remark}

\end{document}